\nonstopmode
\documentclass[10pt]{amsart}
\usepackage{graphicx}
\usepackage{latexsym}
\usepackage{fancyhdr}
\usepackage{amsmath, amssymb, amsthm}
\usepackage[all]{xy}
\usepackage{pdflscape}
\usepackage{longtable}
\usepackage{rotating}
\usepackage{verbatim}
\usepackage{hyperref}
\usepackage{cleveref}
\usepackage{subfigure}
\usepackage{mathrsfs}
\usepackage{mdwlist}
\usepackage{dsfont}
\usepackage{mathtools}
\usepackage{float}
\usepackage{color}
\usepackage{stmaryrd}

\usepackage{tkz-base}
\usepackage{pgfplots}
\pgfplotsset{compat=newest}

\usepackage{tkz-euclide}
\usepackage{etoolbox}

\definecolor{teal}{rgb}{0.0, 0.5, 0.5}

\newcounter{mnotecount}[section]

\newcommand{\rmnote}[1]{}

\allowdisplaybreaks

\DeclareFontFamily{U}{mathb}{\hyphenchar\font45}
\DeclareFontShape{U}{mathb}{m}{n}{
      <5> <6> <7> <8> <9> <10> gen * mathb
      <10.95> mathb10 <12> <14.4> <17.28> <20.74> <24.88> mathb12
      }{}
\DeclareSymbolFont{mathb}{U}{mathb}{m}{n}
\DeclareFontSubstitution{U}{mathb}{m}{n}

\let\dot\relax
\DeclareMathAccent{\dot}{0}{mathb}{"39}
\let\ddot\relax
\DeclareMathAccent{\ddot}{0}{mathb}{"3A}
\let\dddot\relax
\DeclareMathAccent{\dddot}{0}{mathb}{"3B}
\let\ddddot\relax
\DeclareMathAccent{\ddddot}{0}{mathb}{"3C}

\numberwithin{equation}{section}

\theoremstyle{plain}
\newtheorem*{theorem*}{Theorem}
\newtheorem{theorem}{Theorem}[section]
\newtheorem*{lemma*}{Lemma}
\newtheorem{lemma}[theorem]{Lemma}
\newtheorem*{assumption*}{Assumption}

\newtheorem*{proposition*}{Proposition}
\newtheorem{proposition}[theorem]{Proposition}
\newtheorem*{corollary*}{Corollary}
\newtheorem{corollary}[theorem]{Corollary}
\newtheorem*{claim*}{Claim}

\newtheorem*{conjecture*}{Conjecture}

\newtheorem*{question*}{Question}
\newtheorem*{result*}{Result}

\theoremstyle{definition}
\newtheorem*{definition*}{Definition}
\newtheorem{definition}[theorem]{Definition}
\newtheorem*{example*}{Example}

\newtheorem*{algorithm*}{Algorithm}
\newtheorem*{remark*}{Remark}
\newtheorem*{remarks*}{Remarks}
\newtheorem{remark}[theorem]{Remark}

\newtheorem*{convention*}{Convention}

\theoremstyle{plain}

\def\al{\alpha}
\def\be{\beta}
\def\ga{\gamma}
\def\de{\delta}
\def\ep{\epsilon}

\def\rh{\rho}

\def\ta{\tau}

\def\vh{\varphi}
\def\ch{\chi}
\def\ps{\psi}
\def\om{\omega}

\def\Ga{\Gamma}
\def\De{\Delta}

\def\La{\Lambda}

\def\Om{\Omega}

\def\N{\mathbb{N}}

\def\R{\mathbb{R}}

\def\cB{\mathcal{B}}
\def\cC{\mathcal{C}}

\def\cZ{\mathcal{Z}}

\def\sH{\mathscr{H}}

\def\p{\partial}
\def\id{\on{id}}

\def\<{\langle}
\def\>{\rangle}
\renewcommand{\o}{\circ}

\def\ol{\overline}

\def\Hol{\textnormal{H\"ol}}
\def\Zyg{\on{Zyg}}

\let\on=\operatorname

\newcommand{\sr}[1]%
{\ifmmode{}^\dagger\else${}^\dagger$\fi\ifvmode
\vbox to 0pt{\vss
 \hbox to 0pt{\hskip\hsize\hskip1em
 \vbox{\hsize3cm\raggedright\pretolerance10000
 \noindent #1\hfill}\hss}\vss}\else
 \vadjust{\vbox to0pt{\vss%
 \hbox to 0pt{\hskip\hsize\hskip1em%
 \vbox{\hsize3cm\raggedright\pretolerance10000%
 \noindent #1\hfill}\hss}\vss}}\fi%
}

\providecommand{\mapsfrom}{\kern.2em%
\setbox0=\hbox{$\leftarrow$\kern-.10em\rule[0.26mm]{0.1mm}{1.3mm}}\box0%
\kern.3em}

\title[H\"older--Zygmund classes on smooth curves]{H\"older--Zygmund classes on smooth curves}

\author[Armin Rainer]{Armin Rainer}

\address{Fakult\"at f\"ur Mathematik, Universit\"at Wien,
Oskar-Morgenstern-Platz~1, A-1090 Wien, Austria}

\email{armin.rainer@univie.ac.at}

\begin{document}

\begin{abstract}
  We prove that a function in several variables is in the local Zygmund class $\cZ^{m,1}$ if and only if
  its composite with every smooth curve is of class $\cZ^{m,1}$.
  This complements the well-known analogous result for local H\"older--Lipschitz classes $\cC^{m,\al}$
  which we reprove along the way. We demonstrate that these results generalize to mappings between Banach spaces
  and use them to study the regularity of the superposition operator $f_* : g \mapsto f \o g$
  acting on the global Zygmund space $\La_{m+1}(\R^d)$.
  We prove that, for all integers $m,k\ge 1$, the map $f_* : \La_{m+1}(\R^d) \to \La_{m+1}(\R^d)$ is of Lipschitz class $\cC^{k-1,1}$
  if and only if $f \in \cZ^{m+k,1}(\R)$.
\end{abstract}

\thanks{Supported by FWF-Project P 32905-N}
\date{\today}

\maketitle

\section{Introduction}

The main goal of this note is to prove the following $1$-dimensional characterization of
the local \emph{Zygmund class} $\cZ^{m,1}$ of functions in several variables.

\begin{theorem} \label{thm:main1}
  Let $m \in \N$.
  Let $f : U \to \R$ be a function defined on an open subset $U \subseteq \R^d$. The following conditions are equivalent:
  \begin{enumerate}
    \item $f \o c \in \cZ^{m,1}(\R)$ for each $c \in \cC^\infty(\R,U)$.
    \item $f \in \cZ^{m,1}(U)$.
  \end{enumerate}
\end{theorem}

By definition, $\cZ^{m,1}(U)$ consists of all $\cC^{m}$-functions $f : U \to \R$
such that for all compact subsets $K \subseteq U$ and all multiindices
$\ga \in \N^d$ with $|\ga| = m$
the set
\[
  \Big\{\frac{f^{(\ga)}(x+h)- 2 f^{(\ga)}(x)+  f^{(\ga)}(x-h)}{h} : x,x\pm h \in K, \, h \ne 0\Big\}
\]
is bounded.

The local Zygmund classes $\cZ^{m,1}(U)$ fit in the scale of
local \emph{H\"older--Lipschitz classes}:
in fact, $\cC^{m,1}(U) \subseteq \cZ^{m,1}(U)\subseteq \cC^{m,\al}(U)$ for all $m \in \N$ and $\al \in (0,1)$ with strict inclusions.
In some respects (e.g.\ in harmonic analysis, cf.\ \cite{Krantz83})
the Zygmund class $\cZ^{m,1}(U)$ is more natural and important than the
Lipschitz class $\cC^{m,1}(U)$.
We speak of the scale of local \emph{H\"older--Zygmund classes}, where for $\al =1$ the Zygmund class replaces the respective Lipschitz class:
\[
\{\cC^{m,\al}(U): m \in \N,\, \al \in (0,1)\}\cup \{\cZ^{m,1}(U): m \in \N\}.
\]
It corresponds to the scale of global H\"older--Zygmund classes $\{\La_s(\R^d): s >0\}$ which are the Besov classes $\cB^{s,\infty}_\infty(\R^d)$;
see \Cref{sec:definitions} for precise definitions.

The H\"older--Lipschitz analogue of \Cref{thm:main1} is due to Boman \cite{Boman67}:

\begin{theorem} \label{thm:main2}
  Let $m \in \N$ and $\al \in (0,1]$.
  Let $f : U \to \R$ be a function defined on an open subset $U \subseteq \R^d$. The following conditions are equivalent:
  \begin{enumerate}
    \item $f \o c \in \cC^{m,\al}(\R)$ for each $c \in \cC^\infty(\R,U)$.
    \item $f \in \cC^{m,\al}(U)$.
  \end{enumerate}
\end{theorem}

Faure \cite{Faure89} generalized this result to $\cC^{m,\om}$, where $\om$ is any modulus of continuity.

The results of \Cref{thm:main1} and \Cref{thm:main2} have a certain similarity with the well-known fact (e.g.\ \cite[Theorem 9.1]{Krantz83})
that a function $f : \R^d \to \R$ belongs to $\La_s(\R^d)$ if it does so \emph{uniformly} in each variable separately.
There is no uniformity (with respect to $c$) required in condition (1)
of \Cref{thm:main1} and \Cref{thm:main2}, respectively; on the other hand, $f$ is tested on the much larger set of all $\cC^\infty$-curves in the
domain (instead of just affine lines parallel to the axes).

We will prove \Cref{thm:main1} in \Cref{sec:proofs}. Only a few modifications are necessary to obtain also a proof
of \Cref{thm:main2}. The general proof scheme is the one used in \cite[Section 4.3]{FK88} and \cite[Section 12]{KM97} for the Lipschitz case.
It is combined with the characterization of H\"older--Zygmund classes in terms of finite differences; see \cite{Krantz83}.

The characterization in \Cref{thm:main1} enables us to extend the notion of local Zygmund classes to mappings between
convenient vector spaces (i.e.\ Mackey--complete locally convex spaces) such as has been done for H\"older--Lipschitz classes
in \cite{FK88,FF89,KM97}. This leads to a version of \Cref{thm:main1} for maps between Banach spaces; see \Cref{cor:Banach}.

We also discuss versions of \Cref{thm:main1}, where the domain of definition of $f$ is not an open set.
Then a loss of regularity depending on the geometry of the boundary occurs; see \Cref{thm:arcsmooth1},
\Cref{thm:arcsmooth2}, and \Cref{thm:arcsmooth3} as well as \cite{Kriegl97} and \cite{Rainer18,Rainer:2021tr}.

In the final section we utilize a version of \Cref{thm:main2} for maps between Banach spaces
to study the regularity of the nonlinear superposition operator $f_* : g \mapsto f\o g$ acting
on the global Zygmund class $\La_{m+1}(\R^d)$.

\begin{theorem} \label{lefttranslationLip}
  Let $m,k \in \N_{\ge 1}$ and $f : \R \to \R$ a function.
  Then $f_*$ acts on $\La_{m+1}(\R^d)$ and $f_*: \La_{m+1}(\R^d) \rightarrow \La_{m+1}(\R^d)$
  is of Lipschitz class $\cC^{k-1,1}$ if and only if $f \in \cZ^{m+k,1}(\R)$.
\end{theorem}

This complements results of \cite{Bourdaud:2002tb} on the $\cC^k$-regularity of $f_* : \La_{m+1}(\R^d) \to \La_{m+1}(\R^d)$
based on totally different methods.

\subsection*{Notation}
We denote by $\N = \{0,1,2,\ldots\}$ the set of nonnegative integers and set $\N_{\ge k} := \{n \in \N : n\ge k\}$.
We will make use of standard multiindex notation.
The partial derivative of a function $f$ with respect to the $j$-th variable is denoted by $\p_j f$.
If $\ga \in \N^d$ is a multiindex, then we use the notation $f^{(\ga)}  = \p^\ga f = \p_1^{\ga_1} \cdots \p_d^{\ga_d} f$
for the corresponding partial derivative of higher order.

\section{H\"older--Zygmund classes} \label{sec:definitions}

\subsection{Function spaces}

By a \emph{modulus of continuity} we mean
an increasing subadditive function $\om : [0,\infty) \to [0,\infty)$ such that
$\lim_{t \to 0}\om(t) = 0$ and $t \mapsto t/\om(t)$ is locally bounded.

Let $\om$ be a modulus of continuity.
Let $U \subseteq \R^d$ be an open set.
Recall that $\cC^{0,\om}(U)$ denotes the set of functions $f : U \to \R$ such that
\[
  \sup_{x,x+h \in K,\, h \ne 0}\frac{|f(x+h)-f(x)|}{\om(|h|)} < \infty
\]
for all compact subsets $K \subseteq U$.
For a positive integer $m$,
\[
  \cC^{m,\om}(U) := \big\{ f \in \cC^m(U) : f^{(\ga)} \in \cC^{0,\om}(U) \text{ for all } \ga \in \N^d \text{ with } |\ga| = m \big\}.
\]
For $\om(t)=t^\al$, where $\al \in (0,1]$, we obtain the \emph{H\"older--Lipschitz classes} $\cC^{m,\al}(U)$.

The \emph{Zygmund class} $\cZ^{0,1}(U)$ is the set of all continuous functions $f : U \to \R$
such that
\begin{equation} \label{eq:Zdef}
  \sup_{x,x\pm h \in K,\, h \ne 0}\frac{|f(x+h)-2f(x) + f(x-h)|}{|h|} < \infty,
\end{equation}
and if $m$ is a positive integer,
\[
  \cZ^{m,1}(U) := \big\{ f \in \cC^m(U) : f^{(\ga)} \in \cZ^{0,1}(U) \text{ for all } \ga \in \N^d \text{ with } |\ga| = m \big\}.
\]
Continuity of $f$ does not follow from \eqref{eq:Zdef} and has to be imposed (cf.\ \cite[Proposition 2.7]{Krantz83}).
All classes $\cC^{m,\om}(U)$ and $\cZ^{m,1}(U)$ are endowed with their natural locally convex topologies.

Note that for all $m \in \N$ and $0 < \al < \be < 1$, we have the strict continuous inclusions
\begin{align} \label{eq:inclusions}
    \cC^{m+1}(U) \subsetneq \cC^{m,1}(U) \subsetneq \cZ^{m,1}(U) \subsetneq \cC^{m,\om}(U) \subsetneq \cC^{m,\be}(U) \subsetneq \cC^{m,\al}(U) \subsetneq \cC^m(U),
\end{align}
where $\om(t):= t \log \frac{1}{t}$. For instance $\cZ^{0,1}(\R)$ contains the Weierstrass function $t \mapsto \sum_k 2^{-k} \sin(2^k t)$
which is nowhere differentiable and thus not locally Lipschitz.

  The spaces $\cC^{m,\al}(U)$, $0<\al<1$, and $\cZ^{m,1}(U)$ are local versions of the global \emph{H\"older--Zygmund spaces}
  $\La_s(\R^d)$, $s>0$, which are Banach spaces defined as follows.
  For $0 < s \le 1$, $\La_s(\R^d)$ consists of all bounded continuous functions $f : \R^d \to \R$ such that
  $\|f\|_{\La_s} <\infty$, where
  \begin{align*}
      \|f\|_{\La_s}  &:= \sup_{x \in \R^d} |f(x)| + \sup_{x,h \in \R^d,\, h \ne 0} \frac{|f(x+h)-f(x)|}{|h|^s},
      \quad \text{ if } 0 < s< 1,
      \\
      \|f\|_{\La_1}  &:= \sup_{x \in \R^d} |f(x)| + \sup_{x,h \in \R^d,\, h \ne 0} \frac{|f(x+h)-2f(x)+f(x-h)|}{|h|}
  \end{align*}
  For $s>1$ the space $\La_s(\R^d)$ is defined recursively:
  take the unique $m \in \N$ with $m < s \le m+1$.
  Then $\La_s(\R^d)$ consists of all functions $f : \R^d \to \R$ of class $\cC^{m}$ such that
  \[
    \|f\|_{\La_s}  := \|f\|_{\La_{s-1}} + \sum_{j=1}^d \|\p_j f\|_{\La_{s-1}} < \infty.
  \]
  For $0< s < t$ there is a continuous inclusion $\La_t(\R^d) \hookrightarrow \La_s(\R^d)$.
  For integer $s = m+1 >0$, $\La_s(\R^d)$ strictly contains the \emph{Lipschitz space}
  \[
    \on{Lip}_s(\R^d) := \big\{ f \in \cC^{m}(\R^d) : \|f\|_{\on{Lip}_s} < \infty \big\},
  \]
  where (again recursively)
  \begin{align*}
     \|f\|_{\on{Lip}_1}  &:= \sup_{x \in \R^d} |f(x)| + \sup_{x,h \in \R^d,\, h \ne 0} \frac{|f(x+h)-f(x)|}{|h|},
     \\
     \|f\|_{\on{Lip}_s}  &:= \|f\|_{\on{Lip}_{s-1}} + \sum_{j=1}^d \|\p_j f\|_{\on{Lip}_{s-1}} \quad \text{ for } s>1.
  \end{align*}
  Note that $\La_t(\R^d) \hookrightarrow \on{Lip}_s(\R^d) \hookrightarrow \La_s(\R^d)$ if
  $s$ is an integer and $t>s$.

\subsection{Difference quotients and finite differences}

Let $f : \R \to \R$ be a function.
The \emph{difference quotient} $\de^m f(x_0,\ldots,x_m)$ of order $m$ on the pairwise disjoint points $x_0,\ldots,x_m\in \R$
is recursively defined by $\de^0 f(x_0) := f(x_0)$ and
\begin{align*}
  \de^m f(x_0,\ldots,x_m) := m \frac{\de^{m-1} f(x_0,\ldots,x_{m-1}) - \de^{m-1} f(x_1,\ldots,x_m)}{x_0 - x_m}.
\end{align*}
It is symmetric in $x_0,\ldots,x_m$.
One checks easily that
\[
 	\de^m f(x_0,\ldots,x_m) = m! \sum_{i=0}^m f(x_i) \prod_{\substack{0\le j \le m\\ j \ne i}} \frac{1}{x_i - x_j}.
\]
We will mainly use the \emph{equidistant difference quotient}
\begin{align*}
  \de^m_{\text{eq}} f(x;h) &:= \de^m f(x,x+h,\ldots,x+mh)
  \\
  &= \frac{1}{h^m} \sum_{i=0}^m (-1)^{m-i} \binom{m}{i} f(x +ih)
  =: \frac{1}{h^m} \De^m_hf(x),
\end{align*}
where $\De^m_h f(x)$
is the \emph{(forward) finite difference} of order $m$ recursively defined by
\begin{align*}
	\De^0_h f(x) &:= f(x), \quad
	\De^1_h f(x) := f(x+h) - f(x), \quad
  \De^m_h f(x) := \De^1_h (\De^{m-1}_h f(x)).
\end{align*}
Note that, for $f \in \cC^1(\R)$,
\begin{equation} \label{eq:integral}
   \De^2_h f(x) = \int_x^{x+h} \De_h^1 f'(t)\, dt.
\end{equation}

\subsection{Product and chain rule}

Let $f,g,f_i : \R \to \R$.
We have
\begin{align} \label{eq:Leibniz}
   \De^m_h (f_1\cdots f_n)(x)
   &= \sum_{i_1 + \cdots + i_n = m} \binom{m}{i_1,\ldots,i_n} \prod^n_{j=1} \De^{i_j}_h f_j\Big(x + \Big(m - \sum_{k=j}^n i_k\Big)h\Big).
\end{align}
We also need a chain rule for finite differences of order one and two:
  \begin{align} \label{chainrule}
  \begin{split}
    \De^1_h (f \o g)(x) &= \De^1_{\De^1_h g(x)} f(g(x)),
    \\
    \De^2_h (f \o g)(x) &= \De^1_{\De^2_h g(x)} f(2g(x+h)-g(x))
    +
    \De^2_{\De^1_h g(x)} f(g(x)).
  \end{split}
  \end{align}
The validity of these formulas is easily established by expanding the right-hand sides.


\subsection{H\"older--Zygmund classes in terms of difference quotients}
First we recall a description of local Lipschitz classes.

\begin{theorem}[{\cite[Lemma 12.4]{KM97}}] \label{thm:Lchar}
  Let $m \in \N$.
  Let $f : \R \to \R$ be continuous. The following conditions are equivalent:
  \begin{enumerate}
    \item $f \in \cC^{m,1}(\R)$.
    \item $(x,h) \mapsto \de^{m+1}_{\on{eq}} f(x;h)$
    is locally bounded on $\R \times (\R\setminus \{0\})$.
  \end{enumerate}
\end{theorem}

As a consequence  we obtain

\begin{corollary}
  \label{cor:bd}
  Let $m \in \N$.
  Let $f : \R \to \R$ be continuous.
  That $(x,h) \mapsto \de^{m}_{\on{eq}} f(x;h)$ is locally bounded implies that $(x,h) \mapsto \de^{j}_{\on{eq}} f(x;h)$
  is locally bounded for all $j\le m$.
\end{corollary}

To get a similar characterization of local H\"older--Zygmund classes we recall

\begin{theorem}[{\cite[Theorem 6.1]{Krantz83}}] \label{thm:Krantz}
  If $f \in L^\infty(\R) \cap C^0(\R)$ and $0 < s <n$ with $n \in \N$, then
  $f \in \La_s(\R)$ is equivalent to $|\De_h^n f(x)| \le C\, |h|^s$ for all $x,h \in \R$.
\end{theorem}

For local Zygmund classes we may infer

\begin{theorem} \label{thm:Zchar}
  Let $m \in \N$.
  Let $f : \R \to \R$ be continuous. The following conditions are equivalent:
  \begin{enumerate}
    \item $f \in \cZ^{m,1}(\R)$.
    \item $(x,h) \mapsto h\, \de^{m+2}_{\on{eq}} f(x;h)$
    and
    $(x,h) \mapsto \de^{m}_{\on{eq}} f(x;h)$ are locally bounded on $\R \times (\R\setminus \{0\})$.
  \end{enumerate}
\end{theorem}

\begin{proof}
    (1) $\Rightarrow$ (2)
    Let $f  \in \cZ^{m,1}(\R)$ and $I \subseteq \R$ a bounded interval.
    Then $f|_I$ has an extension $F \in \La_{m+1}(\R)$ (such that $\|F\|_{\La_{m+1}(\R)} \le C\, \|f\|_{{\La_{m+1}(I)}}$ for some $C>0$); see \cite[Section 14]{Krantz83}.
    By \Cref{thm:Krantz}, $|\De^{m+2}_h F(x)| \le C\, |h|^{m+1}$ for all $x,h \in \R$.
    From this it is easy to conclude that $(x,h) \mapsto h\, \de^{m+2}_{\on{eq}} f(x;h)$ is locally bounded.
    That also $(x,h) \mapsto \de^{m}_{\on{eq}} f(x;h)$ is locally bounded is trivial for $m=0$
    and follows from \Cref{thm:Lchar} for $m\ge 1$, since
    $f \in \cZ^{m,1}(\R) \subseteq \cC^{m-1,1}(\R)$ by \eqref{eq:inclusions}.

    (2) $\Rightarrow$ (1)
    Fix a bounded interval $I \subseteq \R$ and a $\cC^\infty$-function $\ch : \R \to [0,1]$ with compact support
    which is $1$ in a neighborhood of $I$. Then, by the product rule \eqref{eq:Leibniz},
    $g:= \ch f$ satisfies $|\De_h^{m+2} g(x)| \le C\, |h|^{m+1}$ for all $x,h \in \R$ and thus $g \in \La_{m+1}(\R)$, by \Cref{thm:Krantz}.
    Indeed, $|\De_h^{m+2} g(x)|$ is bounded by a finite sum of terms which are up to constant factors
    of the form $|\De^i_h\ch(y) \De_h^{j} f(z)|$, where $i+j = m+2$.
    If $j\le m$, then $|\De_h^{j} f(z)| \le C |h|^j$ and $|\De_h^{m+1} f(z)| \le C |h|^m$, by \Cref{cor:bd}.
    Thus $|\De^i_h\ch(y) \De_h^{j} f(z)| \le C |h|^{m+1}$ for $j\le m+1$. For $j=m+2$ this follows from local boundedness of
    $(x,h) \mapsto h\, \de^{m+2}_{\on{eq}} f(x;h)$.
    Since $I$ was arbitrary, we may conclude that $f \in \cZ^{m,1}(\R)$.
\end{proof}

\begin{remark}
  The local boundedness of $(x,h) \mapsto \de^{m}_{\on{eq}} f(x;h)$ is used for the ``local to global''
  argument in (2) $\Rightarrow$ (1). It is possible that it can be dropped from the formulation of (2).
\end{remark}

For the local H\"older classes we have

\begin{theorem} \label{thm:Hchar}
  Let $m \in \N$ and $\al \in (0,1)$.
  Let $f : \R \to \R$ be continuous. The following conditions are equivalent:
  \begin{enumerate}
    \item $f \in \cC^{m,\al}(\R)$.
    \item $(x,h) \mapsto |h|^{1-\al} \de^{m+1}_{\on{eq}} f(x;h)$
    and $(x,h) \mapsto  \de^{m}_{\on{eq}} f(x;h)$ are locally bounded on $\R \times (\R\setminus \{0\})$.
  \end{enumerate}
\end{theorem}

\begin{proof}
   This follows in analogy to \Cref{thm:Zchar} again from \Cref{thm:Krantz}.
\end{proof}

\section{Testing on smooth curves} \label{sec:proofs}

This section is devoted to the proof of \Cref{thm:main1}.
A few adjustments are all that is needed to also prove \Cref{thm:main2} in one go.
These adjustments are indicated in \Cref{sec:proofmain2}.

\subsection{Composition with smooth curves preserves the class}

\begin{proposition} \label{prop:composition}
  Let $U \subseteq \R^d$ be an open subset and $c \in \cC^{\infty}(\R,U)$. Then:
  \begin{enumerate}
    \item For each $m \in \N$ and $f \in \cZ^{m,1}(U)$ we have $f \o c\in \cZ^{m,1}(\R)$.
    \item For each $m \in \N$, $\al \in (0,1]$, and $f \in \cC^{m,\al}(U)$ we have $f \o c \in \cC^{m,\al}(\R)$.
  \end{enumerate}
\end{proposition}

The proposition is a consequence of sharper versions: e.g.\ \cite{Norton94}, \cite{Bourdaud:2002tb}, and \cite{LlaveObaya99}.
For $m \ge 1$ a proof of (1) can be assembled from the arguments in \Cref{sec:superposition}; see \Cref{rem:acts}.

\subsection{Curve lemma}
We will repeatedly use the \emph{general curve lemma} \cite[12.2]{KM97} which we restate here in a simple form
for the convenience of the reader.

\begin{lemma}\label{curvelemma}
Let $c_n \in \cC^\infty(\R,\R^d)$ be a sequence of $\cC^\infty$-curves that converges fast to $0$, i.e.,
for each $k \in \N$ the sequence $(n^k c_n)_n$ is bounded in $\cC^\infty(\R,\R^d)$.
Let $s_n \ge 0$ be reals with $\sum_n s_n <\infty$.
Then there exist a curve $c \in \cC^\infty(\R,\R^d)$ and a convergent sequence $t_n$ of reals such that
$c(t+t_n) = c_n(t)$ for $|t|\le s_n$ and all $n$.
\end{lemma}

\subsection{Degree zero}
The proof of \Cref{thm:main1} is based on induction on $m$. The following lemma treats the base case $m=0$.

\begin{lemma} \label{lem:C0om}
  Let $U \subseteq \R^d$ be an open set and $f : U \to \R$ a function. The following conditions are equivalent:
  \begin{enumerate}
    \item $f \o c \in \cZ^{0,1}(\R)$ for all $c  \in \cC^{\infty}(\R,U)$.
    \item $f \in \cZ^{0,1}(U)$.
  \end{enumerate}
\end{lemma}

\begin{proof}
  That (2) implies (1) follows from \Cref{prop:composition}.
  Let us assume that (1) holds
  and suppose for contradiction that $f \not\in \cZ^{0,1}(U)$.
  Note that (1) implies that $f$ is continuous (cf.\ \cite[Theorem 4.11]{KM97}).
  Thus there is a compact set $K \subseteq U$  and points $x_n$, $x_n \pm h_n$ in $K$ such that
  \[
    q_n := \frac{|f(x_n+h_n) - 2 f(x_n) + f(x_n-h_n)|}{|h_n|}
  \]
  is unbounded.
  Passing to subsequences, we may assume that $|x_n-x| \le 4^{-n}$, $|h_n| \le 4^{-n}$, and $q_n \ge n 2^n$.
  Consider the curves $c_n(t) := x_n + t \frac{h_n}{2^n |h_n|}$ and note that $c_n -x$ converges fast to $0$.
  By \Cref{curvelemma}, there is a $\cC^\infty$-curve $c : \R \to U$ and a convergent sequence of reals $t_n$ such that
  $c(t+t_n) = c_n(t)$ for all $|t| \le s_n := 2^n |h_n|$.
  Then
  \begin{align*}
    \frac{|(f\o c)(t_n + s_n) - 2(f\o c)(t_n) + (f\o c)(t_n - s_n)|}{s_n}
    = \frac{q_n}{2^n} \ge n
  \end{align*}
  contradicting (1).
\end{proof}

\begin{lemma} \label{lem:C0omHL}
  Let $\al \in (0,1]$, $U \subseteq \R^d$ an open set,
  and $f : U \to \R$ a function. The following conditions are equivalent:
  \begin{enumerate}
    \item $f \o c \in \cC^{0,\al}(\R)$ for all $c  \in \cC^{\infty}(\R,U)$.
    \item $f \in \cC^{0,\al}(U)$.
  \end{enumerate}
\end{lemma}

\begin{proof}
  Repeat the proof of \Cref{lem:C0om} with $q_n := |f(x_n+h_n)-f(x_n)|/|h_n|^\al$;
  cf.\ \cite{Faure91}, \cite[Lemma 12.7]{KM97}, or \cite{KMR}.
\end{proof}

\subsection{Proof of \Cref{thm:main1}}
The key step is the following proposition.

\begin{proposition} \label{lem:inductionstep}
  Let $m \in \N_{\ge 1}$.
  Let $f : \R^2 \to \R$ be such that $f \o c \in \cZ^{m,1}(\R)$ for all $c \in \cC^\infty(\R,\R^2)$.
  Then $\p_2f(\cdot,0) \in \cZ^{m-1,1}(\R)$.
\end{proposition}

Let us now use \Cref{lem:inductionstep} to complete the proof of \Cref{thm:main1}. The proof of \Cref{lem:inductionstep} will be given in \Cref{sec:proofkey}.

\begin{proposition} \label{prop:key4}
  Let $m \in \N_{\ge 1}$. Let $U \subseteq \R^d$ be an open set.
  Let $f : U \to \R$ be a function such that $f \o c \in \cZ^{m,1}(\R)$ for all $c \in \cC^\infty(\R,U)$.
  Then $d_v f(x) := \p_t|_{t=0} f(x+tv)$ exists for all $(x,v) \in U \times \R^d$ and defines a mapping $df : U \times \R^d \to \R$ such that
    $df \o (x,v) \in  \cZ^{m-1,1}(\R)$ for each $(x,v) \in \cC^\infty(\R,U \times \R^d)$.
\end{proposition}

\begin{proof}
   The directional derivative $d_vf(x)$ exists, since $s \mapsto f(x+sv)$ belongs to $\cZ^{m,1} \subseteq \cC^1$ for $s$ near $0$.
   Let $(x,v) \in \cC^\infty(\R,U \times \R^d)$ and
   consider the $\cC^\infty$-map $g(t,s) := x(t) + s v(t)$.
   Then the open set $\Om := g^{-1}(U)$ contains $\R \times \{0\}$.
   Fix $R,r>0$ such that $[-R,R] \times [-r,r] \subseteq \Om$.
   For any $u>0$ choose a $\cC^\infty$-function $\vh_u : \R \to [-u,u]$ such that $\vh_u(x)=x$ for all $x \in [-\frac{u}2,\frac{u}2]$.
   Then $\tilde g(t,s) := g(\vh_R(t),\vh_r(s))$ maps $\R^2$ to $U$ and coincides with $g$ on $[-\frac{R}2,\frac{R}2] \times [-\frac{r}2,\frac{r}2]$.
   Thus
   $f \o \tilde g : \R^2 \to \R$ has the property that $f \o \tilde g \o c \in \cZ^{m,1}(\R)$ for all $c \in \cC^\infty(\R,\R^2)$, by assumption. By \Cref{lem:inductionstep},
   $t \mapsto  \p_2 (f\o \tilde g)(t,0)$ belongs to $\cZ^{m-1,1}(\R)$. For $t \in [-\frac{R}2,\frac{R}2]$ we have
   \[
      \p_2 (f\o \tilde g)(t,0) = \p_s|_{s=0} \big(f(x(t) + s v(t))\big) = d_{v(t)} f(x(t)) = (df \o (x,v))(t).
   \]
   Since $R>0$ was arbitrary, we may conclude that $df \o (x,v) \in  \cZ^{m-1,1}(\R)$.
\end{proof}

Now we may complete the proof of \Cref{thm:main1}.
One implication simply follows from \Cref{prop:composition}.
For the other implication
suppose that $f : U \to \R$
has the property that
$f \o c \in \cZ^{m,1}(\R)$ for all $c \in \cC^\infty(\R,U)$.
   We proceed by induction on $m$. The case $m=0$ follows from \Cref{lem:C0om}.
   Suppose that $m\ge 1$.
   We may infer from \Cref{prop:key4} that the partial derivatives $\p_j f(x)$, $j=1,\ldots,d$, of first order exist at all $x \in U$
   and $\p_j f \o c \in  \cZ^{m-1,1}(\R)$ for each $c \in C^\infty(\R,U)$ and all $j$.
   The induction hypothesis implies that
   $\p_j f \in  \cZ^{m-1,1}(U)$ for all $j$, that is,
   $f \in \cZ^{m,1}(U)$.
This ends the proof of \Cref{thm:main1}.

\subsection{Auxiliary results}
We need some preparatory results for the proof of  \Cref{lem:inductionstep}.
First we derive some properties of functions in $\cZ^{m,1}(\R)$.

We use the Landau $O$-notation at $0$ in the following way.
Let $I \subseteq \R$ be a bounded interval.
If $\vh_x(h)$ is a function in $h$ which may also depend on $x \in I$ and $\ps(h)>0$ is a function of $h$,
then $\vh_x(h) = O_I(\ps(h))$ shall mean that there is a constant $C=C(I)>0$ such that
$|\vh_x(h)| \le C\,\ps(h)$ for all $x \in I$ and all sufficiently small $h$.

Of crucial importance will be a Taylor formula for functions in $\cZ^{m,1}(\R)$:

\begin{lemma} \label{lem:Taylor}
  Let $m \in \N$ and $f \in \cZ^{m,1}(\R)$. Then, for each bounded interval $I \subseteq \R$,
  \begin{equation} \label{eq:Taylor}
    \frac{1}{2^m} f(x+2h) - 2 f(x+h)
    + \sum_{j=0}^m \Big(2 - \frac{1}{2^{m-j}}\Big) \frac{f^{(j)}(x)}{j!}  h^{j} = O_I (|h|^{m+1}).
  \end{equation}
\end{lemma}

\begin{proof}
   We proceed by induction on $m$. For $m=0$ the statement follows from the definition.
   Let us assume that the identity holds for $m$ and show it for $m+1$. We suppose that $h>0$;
   if $h<0$ the arguments are similar.
   Let $F \in  \cZ^{m+1,1}(\R)$ with $F' = f \in  \cZ^{m,1}(\R)$. Integrating \eqref{eq:Taylor} in $h$ yields
   \begin{align*}
     \frac{1}{2^{m+1}}&\big(F(x+2h) - F(x)\big)- 2 \big(F(x+h) - F(x)\big)
      + \sum_{j=0}^m \Big(2 - \frac{1}{2^{m-j}}\Big) \frac{F^{(j+1)}(x)}{(j+1)!}  h^{j+1}
      \\
      &=\frac{1}{2^{m+1}}F(x+2h) - 2 F(x+h)
       + \sum_{i=0}^{m+1} \Big(2 - \frac{1}{2^{m+1-i}}\Big) \frac{F^{(i)}(x)}{i!}  h^{i}
      \\
      &= O_I(h^{m+2})
   \end{align*}
   completing the induction.
\end{proof}

Let $f : \R \to \R$ and $m \in \N$. We set
\[
  D_m(f)(x;h) := \frac{1}{2^m} f(x+2h) - 2 f(x+h),
  \quad x,h \in \R.
\]
If $m\ge 1$ we can approximate $hf'(x)$ to order $m+1$ in $h$
by a suitable linear combination of $D_m(f)(x;jh)$, for $j = 0,1,\ldots,m$,
in a uniform way for $f \in \cZ^{m,1}(\R)$ and $x \in I$:

\begin{lemma} \label{lem:firstder}
  Let $m \in \N_{\ge 1}$. There exist constants $a_0,\ldots,a_{m} \in \R$ such that for all $f \in \cZ^{m,1}(\R)$ we have,
  for each bounded interval $I \subseteq \R$,
    \[
      h f'(x) -  \sum_{j=0}^{m}  a_j D_m(f)(x;jh)  =O_I(|h|^{m+1}).
    \]
\end{lemma}

\begin{proof}
  By \Cref{lem:Taylor}, for any choice of $a_0,\ldots,a_{m} \in \R$ we have
  \begin{align*}
    \sum_{j=0}^{m}  a_j D_m(f)(x;jh)
    &= \sum_{j=0}^{m}  a_j \Big(\sum_{i=0}^m \Big(\frac{1}{2^{m-i}}-2\Big) \frac{f^{(i)}(x)}{i!}  (jh)^{i} + O_I(|h|^{m+1}) \Big)
    \\
    &= \sum_{i=0}^m \Big(\sum_{j=0}^{m}  a_j  j^i\Big)  \Big(\frac{1}{2^{m-i}}-2\Big) \frac{f^{(i)}(x)}{i!}  h^{i} + O_I(|h|^{m+1}).
  \end{align*}
  To obtain the assertion it suffices to choose $a_0,\ldots,a_{m}$ such that
  \begin{align*}
     \sum_{j=0}^{m}a_j j &= \Big(\frac{1}{2^{m-1}}-2\Big)^{-1},
     \\
     \sum_{j=0}^{m}a_j j^{i} &= 0, \quad 0\le i \le m,~ i \ne 1,
  \end{align*}
  which is possible, since the coefficients of this linear system of equations form a
  Vandermonde matrix.
\end{proof}

Let us set
\[
    A_m(f)(x;h) := \sum_{j=0}^{m}  a_j D_m(f)(x;jh),
\]
where $a_0,\ldots,a_{m} \in \R$ are the constants provided by \Cref{lem:firstder}.

\subsection{Proof of \Cref{lem:inductionstep}} \label{sec:proofkey}

For $f : \R^2 \to \R$ we define
\[
    \mathbf{D}_m(f)(x,h) := \frac{1}{2^m} f(x,2h) - 2 f(x,h)
\]
and
\[
    \mathbf{A}_m(f)(x,h) := \sum_{j=0}^{m}  a_j \mathbf{D}_m(f)(x,jh),
\]
where $a_0,\ldots,a_{m} \in \R$ are the constants provided by \Cref{lem:firstder}.

From \Cref{lem:firstder} we get an approximation result for functions in two variables:

\begin{lemma} \label{lem:error}
  Let $m \in \N_{\ge 1}$.
  Let $f : \R^2 \to \R$ be such that $f \o c \in \cZ^{m,1}(\R)$ for all $c \in \cC^\infty(\R,\R^2)$.
  For each compact interval $I \subseteq \R$ there is a constant $C>0$ such that for all $x,h \in I$ we have
  \begin{align*}
     \big| h \p_2 f(x,0) - \mathbf{A}_m(f)(x,h) \big| \le C |h|^{m+1}.
  \end{align*}
\end{lemma}

\begin{proof}
  In the case $0 \not \in I$ so that $h$ is bounded away from $0$, it is enough to check that $x \mapsto \p_2 f(x,0)$ is bounded on $I$.
  (That $f$ and consequently $(x,h) \mapsto \mathbf{A}_m(f)(x,h)$ is bounded on $I \times I$ follows easily
  from \Cref{curvelemma} or \cite[2.8]{KM97}.)
  Suppose, for contradiction, that there exist $x_n,x \in I$ with $|x_n - x| \le 4^{-n}$ and $|\p_2 f(x_n,0)| \ge n 2^n$.
  Applying \Cref{curvelemma} to $c_n(t) := (x_n,2^{-n}t)$ and $s_n :=2^{-n}$ gives a $\cC^\infty$-curve $c : \R \to \R^2$
  and a convergent sequence of reals $t_n$
  such that $c(t+t_n)=c_n(t)$ for $|t|\le s_n$.
  Thus,
  $|(f\o c)'(t_n)| = |(f\o c_n)'(0)| = 2^{-n}|\p_2 f(x_n,0)|\ge n$, contradicting that $f\o c \in \cZ^{m,1}(\R) \subseteq \cC^1(\R)$.

  Now assume that $0 \in I$.
  Suppose, for contradiction, that there are $x_n,h_n \in I$ such that
  \begin{align*}
     \big| h_n \p_2 f(x_n,0) -  \mathbf{A}_m(f)(x_n,h_n) \big| \ge n 2^{n(m+1)} |h_n|^{m+1}.
  \end{align*}
  By passing to subsequences, we may assume that $x_n \to x$ and $h_n \to 0$ (by the first paragraph)
  and in turn that
  $|x_n - x| \le 4^{-n}$ and $|h_n| \le 4^{-n}$.
  Applying \Cref{curvelemma} to $c_n(t) := (x_n, 2^{-n}t)$ and $s_n = 2m\cdot 2^{-n}$
  we find a $\cC^\infty$-curve $c : \R \to \R^2$
  and a convergent sequence of reals $t_n$
  with $c(t+t_n) = c_n(t)$ for $|t|\le s_n$.
  Then
  \begin{align*}
    \MoveEqLeft
     \Big|2^n h_n (f \o c)'(t_n) -  A_m(f\o c)(t_n; 2^n h_n) \Big|
     \\
     &= \Big|2^n h_n (f \o c_n)'(0) - A_m(f\o c_n)(0; 2^n h_n) \Big|
     \\
     &= \big| h_n \p_2 f(x_n,0) -  \mathbf{A}_m(f)(x_n,h_n) \big|
      \\
     &\ge n (2^{n} |h_n|)^{m+1}.
  \end{align*}
  But this contradicts \Cref{lem:firstder}.
\end{proof}

Now we are ready to show \Cref{lem:inductionstep}, assuming the validity of \Cref{thm:main2} which will be (re)proved in \Cref{sec:proofmain2}.
Let $m \in \N_{\ge 1}$ and suppose that $f : \R^2 \to \R$ satisfies $f \o c \in \cZ^{m,1}(\R)$ for all $c \in \cC^\infty(\R,\R^2)$.
Then $g := \p_2f(\cdot,0)$ is well-defined. We have to prove that $g \in \cZ^{m-1,1}(\R)$.
By \Cref{thm:Zchar}, it suffices to check the following three claims.

\begin{claim*}[i]
  $g$ is continuous.
\end{claim*}

Since $f \o c \in \cC^{1,\al}(\R)$ for all $c \in \cC^\infty(\R,\R^2)$ and all $\al \in (0,1)$, by \eqref{eq:inclusions},
we may invoke the result for $\cC^{1,\al}$; cf.\ \Cref{thm:main2} and its proof in \Cref{sec:proofmain2}. Claim (i) follows.

\begin{claim*}[ii]
  $\de^{m-1}_{\on{eq}} g(x;h)$ is locally bounded in $(x,h) \in \R \times (\R\setminus \{0\})$.
\end{claim*}

By \eqref{eq:inclusions}, we have
$f \o c \in \cC^{m-1,1}(\R)$ for all $c \in \cC^\infty(\R,\R^2)$.
By \Cref{thm:main2}, $g \in \cC^{m-2,1}(\R)$
   if $m \ge 2$ and thus $\de^{m-1}_{\on{eq}} g(x;h)$ is locally bounded in $(x,h) \in \R \times (\R\setminus \{0\})$, in view of \Cref{thm:Lchar}.
   If $m = 1$ this is trivially true by Claim (i).

\begin{claim*}[iii]
  $h \de^{m+1}_{\on{eq}} g(x;h)$ is locally bounded in $(x,h) \in \R \times (\R\setminus \{0\})$.
\end{claim*}

   Suppose, for contradiction, that Claim (iii) is not true, say for $x$ in a neighborhood of $0$.
   Then there exist $x_n$ and $h_n$ with $|x_n| \le 4^{-n}$ and $0<h_n < 4^{-n}$ (if necessary replace $f(x,y)$ by $f(-x,y)$) such that
   \[
    |h_n \de^{m+1}_{\text{eq}} g(x_n;h_n)| \ge n 2^{n(m+1)}.
   \]
   Let $c_n(t) := (x_n - h_n + 2^{-n}t,2^{-n}t)$ and $s_n := (m+2)2^{-n}$.
   \Cref{curvelemma} gives a $\cC^\infty$-curve $c : \R \to \R^2$
   and a convergent sequence of reals $t_n$
   with $c(t+t_n) = c_n(t)$ for $|t|\le s_n$.
   Set
   \begin{align*}
      f_1(x,h) &:= h g(x),
      \\
      f_2(x,h) &:= h g(x)-\mathbf{A}_m(f)(x,h).
   \end{align*}
   Then the sequence
   \begin{align} \label{eq:Tn}
      T_n :&= 2^n h_n \de^{m+2}_{\on{eq}} (f_1 \o c)(t_n;2^n h_n)
      \\
      &=   \frac{2^n h_n}{(2^n h_n)^{m+2}} \sum_{i=0}^{m+2} (-1)^{m+2-i} \binom{m+2}{i} f_1(x_n -h_n + i h_n, i h_n)
      \notag  \\
      &=   \frac{1}{(2^n h_n)^{m+1}} \sum_{i=0}^{m+2} (-1)^{m+2-i} \binom{m+2}{i} ih_n g(x_n + (i-1) h_n)
      \notag \\
      &=  \frac{1}{2^{n(m+1)}} \frac{h_n}{h_n^{m+1}} \sum_{i=1}^{m+2} (-1)^{m+2-i} \binom{m+2}{i} i g(x_n + (i-1) h_n)
      \notag \\
      &=  \frac{m+2}{2^{n(m+1)}} \frac{h_n}{h_n^{m+1}} \sum_{j=0}^{m+1} (-1)^{m+1-j} \binom{m+1}{j}  g(x_n + j h_n)
      \notag \\
      &=  \frac{m+2}{2^{n(m+1)}} h_n \de^{m+1}_{\on{eq}} g(x_n;h_n)
      \notag
   \end{align}
   satisfies $|T_n| \ge  (m+2) n$.
   On the other hand,
   $\textbf{A}_m(f) \o c \in \cZ^{m,1}(\R)$, by the assumption on $f$,
   so that $2^n h_n \de^{m+2}_{\on{eq}} (\textbf{A}_m(f) \o c)(t_n;2^n h_n)$ is bounded, by \Cref{thm:Zchar}.
   By \Cref{lem:error}, there is $C>0$ such that $|f_2(x,h)| \le C |h|^{m+1}$ for all $x,h \in [-(m+2),m+2]$ which implies that
  \begin{align*}
     2^n h_n |\de^{m+2}_{\on{eq}} (f_2 \o c)(t_n;2^n h_n)|
     &\le
     \frac{2^n h_n}{(2^n h_n)^{m+2}} \sum_{i=0}^{m+2}\binom{m+2}{i} |f_2(x_n -h_n + i h_n, i h_n)|
     \\
     &\le C
     \Big(\frac{m+2}{2^n} \Big)^{m+1} \sum_{i=0}^{m+2}\binom{m+2}{i}.
  \end{align*}
  Consequently, $T_n$ is bounded, a contradiction.
  Thus, Claim (iii) is shown and the proof of \Cref{lem:inductionstep} is complete.

\subsection{Proof of \Cref{thm:main2}} \label{sec:proofmain2}

We will indicate how to show

\begin{proposition} \label{lem:inductionstepHL}
  Let $m \in \N$ and $\al \in (0,1]$.
  Let $f : \R^2 \to \R$ be such that $f \o c \in \cC^{m+1,\al}(\R)$ for all $c \in \cC^\infty(\R,\R^2)$.
  Then $\p_2f(\cdot,0) \in \cC^{m,\al}(\R)$.
\end{proposition}

Then it is easy to finish the proof of \Cref{thm:main2}
using a variant of \Cref{prop:key4} as well as \Cref{lem:C0omHL} and \Cref{prop:composition}.

We replace \Cref{lem:Taylor} by
the following easy consequence of Taylor's formula.

\begin{lemma} \label{lem:TaylorHL}
  Let $m \in \N$, $\al \in (0,1]$, and $f \in C^{m,\al}(\R)$. Then, for each bounded interval $I \subseteq \R$,
  \[
    f(x+h) =  \sum_{j=0}^m \frac{f^{(j)}(x)}{j!} h^j + O_I(|h|^{m+\al}).
  \]
\end{lemma}

As in \Cref{lem:firstder} we conclude:
If $b_0,\ldots, b_m \in\ \R$ is the solution of the system
\begin{align*}
   \sum_{j=0}^{m}b_j j &= 1,
   \\
   \sum_{j=0}^{m}b_j j^{i} &= 0, \quad 0\le i \le m,~ i \ne 1,
\end{align*}
then for all $f \in \cC^{m,\al}(\R)$, $m\ge 1$, and for each bounded interval $I \subseteq \R$,
\begin{equation} \label{eq:firstderHL}
  h f'(x) =  \sum_{j=0}^{m}  b_j f(x+jh)  + O_I(|h|^{m+\al}).
\end{equation}

Now suppose that $m\in \N$, $\al \in (0,1]$, and $f : \R^2 \to \R$ is such that $f\o c \in \cC^{m+1,\al}(\R)$ for all $c \in \cC^\infty(\R,\R^2)$.
Then we infer from \eqref{eq:firstderHL}, in analogy to \Cref{lem:error}, that on each compact interval $I$ there is $C>0$ such that
\begin{equation} \label{eq:errorHL}
  |h g(x) - \mathbf{B}_m(f)(x,h)| \le C |h|^{m+1+\al}, \quad x,h \in I,
\end{equation}
where $g := \p_2 f(\cdot,0)$ and $\mathbf{B}_m(f)(x,h) := \sum_{j=0}^{m}  b_j f(x,jh)$.

To complete the proof of \Cref{lem:inductionstepHL} we have to show that $g \in \cC^{m,\al}(\R)$.
By \Cref{thm:Hchar}, it is enough to check the following three claims.

\begin{claim*}[I]
  $g$ is continuous.
\end{claim*}

To see that $g$ is continuous it suffices to finish the proof of \Cref{lem:inductionstepHL} in the case $m=0$.
\Cref{thm:Hchar} is trivial for $m=0$ and holds without the a priori assumption that $g$ is continuous:
local boundedness of
$|h|^{1-\al} \de^{1}_{\on{eq}} g(x;h) =  |h|^{-\al} (g(x+h)-g(x))$ in $x$ and $h$ is equivalent to $g \in \cC^{0,\al}$.
That means for $m=0$ only Claim (III) must be shown.

\begin{claim*}[II]
    $\de^{m}_{\on{eq}} g(x;h)$ is locally bounded in $(x,h) \in \R \times (\R\setminus \{0\})$.
\end{claim*}

By \eqref{eq:inclusions}, we have $f\o c \in \cC^{m,1}(\R)$ for all $c \in \cC^\infty(\R,\R^2)$.
Thus it suffices to finish the proof of \Cref{lem:inductionstepHL} in the case $\al = 1$.
Indeed, by \Cref{thm:Lchar}, Claim (II) holds if and only if $g \in \cC^{m-1,1}(\R)$
as $g$ is continuous by Claim (I).

\begin{claim*}[III]
  $|h|^{1-\al} |\de^{m+1}_{\text{eq}} g(x;h)|$  is locally bounded in $(x,h) \in \R \times (\R\setminus \{0\})$.
\end{claim*}

We may assume that there exist $x_n$ and $h_n$ with $|x_n| \le 4^{-n}$ and $0<h_n < 4^{-n}$ such that
\[
 h_n^{1-\al} |\de^{m+1}_{\text{eq}} g(x_n;h_n)| \ge n 2^{n(m+1+\al)}.
\]
With the same choices of $c_n$ and $s_n$ as in \Cref{sec:proofkey} we find a $\cC^\infty$-curve $c$ and a convergent sequence of reals $t_n$ such that $c(t+t_n) = c_n(t)$
for $|t|\le s_n$.
Using $f_1(x,h) := h g(x)$ and $T_n$ from \eqref{eq:Tn}, we find
\begin{align*}
   (2^n h_n)^{-\al} |T_n| &= (2^n h_n)^{1-\al} |\de^{m+2}_{\on{eq}} (f_1 \o c)(t_n;2^n h_n)|
   \\
   &=  \frac{m+2}{2^{n(m+1+\al)}} h_n^{1-\al} |\de^{m+1}_{\on{eq}} g(x_n;h_n)| \ge (m+2) n.
\end{align*}
On the other hand, \eqref{eq:errorHL} implies that for $f_2 := f_1 -\mathbf{B}_m(f)$ the sequence
$(2^n h_n)^{1-\al} |\de^{m+2}_{\on{eq}} (f_2 \o c)(t_n;2^n h_n)|$ is bounded.
Since $\mathbf{B}_m(f) \o c \in \cC^{m+1,\al}(\R)$ also
$(2^n h_n)^{1-\al} |\de^{m+2}_{\on{eq}} (\mathbf{B}_m(f) \o c)(t_n;2^n h_n)|$ is bounded, by \Cref{thm:Hchar},
so that $(2^n h_n)^{-\al} |T_n|$ is bounded, a contradiction.

\section{On Banach spaces and beyond} \label{sec:Banach}

In this section we use the characterization in \Cref{thm:main1} to extend the local Zygmund classes to convenient vector spaces.
We shall see in \Cref{cor:Banach} that a version of \Cref{thm:main1} holds for maps between Banach spaces.
For background on convenient analysis we refer to \cite{KM97} and also \cite{FK88}.
We will also discuss versions of \Cref{thm:main1} and \Cref{thm:main2} where the domain of $f$ is not an open set.

\subsection{Convenient analysis of local H\"older--Zygmund classes}
Recall that a \emph{convenient vector space} $E$ is a Mackey-complete locally convex space.
The \emph{$c^\infty$-topology} on $E$ is the final topology with respect to all $\cC^\infty$-curves
(equivalently, all Mackey-convergent sequences) in $E$;
it is not a linear topology.

Let $m \in \N$ and $\al \in (0,1]$.

\begin{definition}
  Let $E, F$ be convenient vector spaces, $U$ a $c^\infty$-open subset of $E$.
  A map $f : U \to F$ is said to be of class $\Zyg^m$ (resp.\ $\Hol^m_\al$) if
  for each $c \in \cC^\infty(\R,U)$ and each $\ell \in F'$ the composite $\ell \o f \o c$
  belongs to $\cZ^{m,1}(\R)$ (resp.\ $\cC^{m,\al}(\R)$).
\end{definition}

For $\Hol^m_\al$ the results we are aiming for have already be established in \cite{FF89} (for $\al=1$ also in \cite{FK88}).
So we will exploit the fact that, by \eqref{eq:inclusions}, if $f$ is of class $\Zyg^m$ then it is of class $\Hol^m_\al$ for any $\al \in (0,1)$.

Let $f : U \to F$ be of class $\Zyg^1$. Then, by \cite[Lemma 7]{FF89}, $f$ is \emph{weakly differentiable},
i.e., for all $x \in U$, $v \in E$, the limit
\[
  df(x,v) := \lim_{t \to 0} \frac{f(x+tv)-f(x)}{t}
\]
exists with respect to the weak topology. By \cite[Proposition 9]{FF89}, $f$ is also \emph{strictly differentiable},
i.e., for each Mackey-compact $K\subseteq U$ and bounded $B \subseteq E$,
\[
  \frac{f(x+sv)-f(x+tv)}{s-t}
\]
is Mackey-convergent to $df(x,v)$ as $s,t\to 0$ uniformly for $x \in K$ and $v \in B$
and $f'(x) := df(x,\cdot) \in L(E,F)$ for every $x \in U$.

\begin{lemma} \label{lem:Banach}
  Let $m \in \N$.
  Let $f : U \to F$ be of class $\Zyg^{m+1}$. Then:
  \begin{enumerate}
      \item $df : U \times E \to F$ is of class $\Zyg^m$.
      \item $f' : U \to L(E,F)$ is of class $\Zyg^m$.
  \end{enumerate}
\end{lemma}

\begin{proof}
  (1) follows from the proof of \Cref{prop:key4}: consider $(t,s) \mapsto \ell(f(x(t)+sv(t)))$ where $\ell \in F'$.
  By the uniform boundedness principle,
  to see (2) it suffices to check that $\on{ev}_v \o f' : U \to F$ is of class $\Zyg^m$ for all $v \in E$ which
  holds by  (1) since $\on{ev}_v \o f' = df(\cdot,v)$.
  Indeed, a curve $c : \R \to G$ in a convenient vector space $G$ is of class $\Zyg^m$ if and only if
  $h \de^{m+2}_{\on{eq}} c(t;h)$ and $\de^{m}_{\on{eq}} c(t;h)$ are locally bounded in
  $(t,h) \in \R \times (\R\setminus \{0\})$, by \Cref{thm:Zchar} and testing with $\ell \in G'$.
\end{proof}

Assume that $f: U \to F$ is weakly differentiable and $df$ is of class $\Zyg^0$.
Let $c \in \cC^\infty(\R, U)$.
Then $f\o c$ is differentiable and $(f\o c)'(t)=df(c(t),c'(t))$,
by \cite[Proposition 11]{FF89}.

\begin{theorem}
  Let $m \in \N$.
  Let $E, F$ be convenient vector spaces, $U$ a $c^\infty$-open subset of $E$, and
  $f : U \to F$ a map.
  The following conditions are equivalent:
  \begin{enumerate}
    \item $f$ is of class $\Zyg^{m+1}$.
    \item $f$ is strictly differentiable and $f' : U \mapsto L(E,F)$ is of class $\Zyg^{m}$.
    \item $f$ is weakly differentiable and $df : U \times E \to F$ is of class $\Zyg^m$.
  \end{enumerate}
\end{theorem}

\begin{proof}
  (1) $\Rightarrow$ (2) follows from \Cref{lem:Banach} and the preceding remarks.

  (2) $\Rightarrow$ (3) The map $f' \times \id : U\times E \mapsto L(E,F)\times E$ is of class $\Zyg^{m}$.
  The evaluation map $\on{ev}: L(E,F)\times E \to F$ is bilinear and smooth. It follows that
  $df = \on{ev} \o (f' \times \id)$ is of class $\Zyg^m$.

  (3) $\Rightarrow$ (1) Let $c \in \cC^\infty(\R, U)$. Then $(f\o c)'(t)=df(c(t),c'(t))$ (see remark before the theorem)
  is of class $\Zyg^m$, whence $f$ is of class $\Zyg^{m+1}$.
\end{proof}

\begin{corollary} \label{cor:Banach}
    Let $m \in \N$.
    Let $E,F$ be Banach spaces, $U$ open in $E$, and $f : U \to F$ a map.
    Then $f$ is of class $\Zyg^m$ if and only if $f$ is $m$-times Fr\'echet differentiable
    such that $f^{(m)} \in \cZ^{0,1}(U,L^m(E,F))$.
\end{corollary}

It is straightforward to adapt the definition of $\cZ^{0,1}$ to maps between Banach spaces.
The proof of \Cref{lem:C0om} shows that such a map is of class $\cZ^{0,1}$ if and only if it is of class $\Zyg^0$;
note that \Cref{curvelemma} is valid in convenient vector spaces.
So, for maps between Banach spaces, $\Zyg^m$ coincides with the naive notion of local Zygmund regularity.

The definition of difference quotient and finite difference of arbitrary order obviously makes sense for
functions $f : \R \to E$ with values in a vector space $E$.
It turns out that \Cref{thm:Zchar} (as well as \Cref{thm:Lchar} and \Cref{thm:Hchar}) are still valid if $E$ is a convenient vector space:

\begin{theorem} \label{thm:ZcharE}
  Let $m \in \N$.
  Let $E$ be a convenient vector space.
  Let $f : \R \to E$ be a function such that $\ell \o f$ is continuous for all $\ell \in E'$.
  The following conditions are equivalent:
  \begin{enumerate}
    \item $f$ is of class $\Zyg^m$.
    \item $(x,h) \mapsto h \de^{m+2}_{\on{eq}} f(x;h)$ and $(x,h) \mapsto \de^{m}_{\on{eq}} f(x;h)$ are bounded on bounded subsets
    of $\R \times (\R\setminus \{0\})$.
  \end{enumerate}
\end{theorem}

\begin{proof}
   This is immediate from \Cref{thm:Zchar}, since (1) and (2) can be tested by composing with $\ell \in E'$.
\end{proof}

\begin{remark}
  Let $E,F$ be convenient vector spaces and $U \subseteq E$ a $c^\infty$-open subset.
  Let us endow $\Zyg^m(\R,F)$ with the
  initial structure with respect to
  all maps $c \mapsto ((x,h)\mapsto  \de^{j}_{\on{eq}} c(x;h))$, for $j=0,1,\ldots,m$,
  and $c \mapsto ((x,h)\mapsto  h \de^{m+2}_{\on{eq}} c(x;h))$
  into the space of all maps
  $\R \times (\R \setminus \{0\}) \to F$ that are bounded on bounded sets, where the latter
  space carries the locally convex topology of uniform convergence on bounded sets.
  Then the space $\Zyg^m(U,F)$ of all maps $f : U \to F$ of class $\Zyg^m$
  endowed with the initial structure with respect to all maps $c^* : \Zyg^m(U,F) \to \Zyg^m(\R,F)$
  for $c\in \cC^\infty(\R,U)$, is a convenient vector space
  and it satisfies the uniform boundedness principle with respect to the
  point evaluations $\on{ev}_x : \Zyg^m(U,F) \to F$ for $x \in U$.
  This can be seen in analogy to \cite[12.11]{KM97}.
\end{remark}

\subsection{Functions on non-open domains}

Let $E,F$ be convenient vector spaces and let $X \subseteq E$ be a convex subset with non-empty $c^\infty$-interior $X^\o$.
We say that a map $f : X \to F$ is of class $\Zyg^{m}$ (resp.\ $\Hol^m_\al$) if
for each $c \in \cC^\infty(\R,X)$ (i.e. $c \in \cC^\infty(\R,E)$ with $c(\R) \subseteq X$)
and each $\ell \in F'$ the composite $\ell \o f \o c$
belongs to $\cZ^{m,1}(\R)$ (resp.\ $\cC^{m,\al}$).

\begin{theorem} \label{thm:arcsmooth1}
Let $E,F$ be convenient vector spaces and let $X \subseteq E$ be a convex subset with non-empty $c^\infty$-interior $X^\o$.
Let $f : X \to F$ be of class $\Zyg^{2m}$ (resp.\ $\Hol^{2m}_\al$) for $m \in \N_{\ge 1}$ (and $\al \in (0,1]$).
Then $f|_{X^\o}$ is of class $\Zyg^{2m}$ (resp.\ $\Hol^{2m}_\al$) and
for $j\le m$ the derivatives $(f|_{X^\o})^{(j)}$ extend uniquely to maps $f^{(j)} : X \to L^j(E,F)$
of class $\Zyg^{2(m-j)}$ (resp.\ $\Hol^{2(m-j)}_\al$).
\end{theorem}

\begin{proof}
   This can be shown in analogy to \cite{Kriegl97} and \cite[Theorem 24.5]{KM97}.
   The crucial ingredient is an application of \Cref{thm:main1} (resp.\ \Cref{thm:main2}) in dimension two.
\end{proof}

In finite dimensions this has been generalized in the $\cC^\infty$- and the $\Hol^m_\al$-setting
by \cite{Rainer18,Rainer:2021tr} to
a large class of closed sets $X \subseteq \R^d$ with $X = \ol{X^\o}$ admitting cusps.
Note that in $\R^d$ the $c^\infty$-topology coincides with the classical topology.
In the $\Hol^m_\al$-setting, a loss of regularity becomes apparent
which is directly related to the sharpness of the cusps.
More precisely, for $\be \in (0,1]$ let $\sH^\be(\R^d)$ denote the family of closed subsets $X  \subseteq \R^d$ with $X = \ol{X^\o}$
such that $X^\o$
has the \emph{uniform $\be$-cusp property}:
  for each $x \in \p X$ there exist $\ep>0$, a cusp
  \[
    \Ga = \Big\{(x',x_d) \in \R^{d-1}\times \R: |x'| <r,\, h \Big(\frac{|x'|}{r}\Big)^\be< x_d < h \Big\}
  \]
  for some $r,h>0$,
  and an orthogonal linear map $A : \R^d \to \R^d$ such that $y + A \Ga \subseteq X^\o$ for all $y \in X \cap B(x,\ep)$.
Consider the functions $p,q : (0,1] \to \N$ defined by
\[
    p(\be) := \Big\lceil \frac{2}{\be} \Big\rceil  \quad \text{ and }\quad q(\be) := \Big\lceil \frac{1}{\be} \Big\rceil.
\]

\begin{theorem}[{\cite[Theorem A]{Rainer:2021tr}}] \label{thm:arcsmooth2}
  Let $m \in \N$ and $\al,\be \in (0,1]$.
  Let $X \in \sH^\be(\R^d)$.
  If $f : X \to \R$ is of class $\Hol^{m p(\be)}_\al$, then all partial derivatives of $f$ of order $j \le m$ extend continuously
  from $X^\o$ to $X$ and are of class $\Hol^{(m-j) p(\be)}_\al$,
  and the partial derivatives of order $m$ are locally $\frac{\al\be}{2q(\be)}$-H\"older continuous on $X$.
\end{theorem}

Combining \Cref{thm:main1} with the proof of \cite[Proposition 3.3 and 3.4]{Rainer:2021tr} gives

\begin{theorem} \label{thm:arcsmooth3}
  Let $m \in \N$ and $\be \in (0,1]$.
  Let $X \in \sH^\be(\R^d)$.
  If $f : X \to \R$ is of class $\Zyg^{mp(\be)}$, then all partial derivatives of $f$ of order $j \le m$ extend continuously
  from $X^\o$ to $X$ and are of class $\Zyg^{(m-j)p(\be)}$.
\end{theorem}

By the inclusion $\cZ^{0,1} \subseteq \cC^{0,\al}$ for each $\al \in (0,1)$ (cf.\ \eqref{eq:inclusions}), also in this case
the partial derivatives of order $m$ satisfy
a local $\frac{\al\be}{2q(\be)}$-H\"older condition on $X$ for each $\al \in (0,1)$.

\section{Regularity of superposition on Zygmund spaces} \label{sec:superposition}

The goal of this section is to prove \Cref{lefttranslationLip}
which characterizes the Lipschitz regularity of the superposition operator
\[
  f_* : g \mapsto f \o g
\]
acting on the global Zygmund spaces $\La_{m+1}(\R^d)$ for $m\in \N_{\ge 1}$.

To put our result in perspective we recall the characterization of the $\cC^{k}$-regularity of $f_*$:
Let $m \in \N_{\ge 1}$ and $f : \R \to \R$ a function. Then:
\begin{enumerate}
  \item $f_* \La_{m+1}(\R^d) \subseteq \La_{m+1}(\R^d)$ if and only if $f \in \cZ^{m,1}(\R)$;
  see \cite[Theorem 1]{Bourdaud:2002tb} or \Cref{rem:acts}.
  \item For $k \in \N$
  the map $f_* : \La_{m+1}(\R^d) \to \La_{m+1}(\R^d)$ is of class $\cC^k$
  if and only if
  $f \in \cC^{m+k}(\R)$ and
  \begin{equation*}
     f^{(m+k)}(x+h)- 2f^{(m+k)}(x) + f^{(m+k)}(x-h)  = o(h) \quad \text { as } h \to 0^+
  \end{equation*}
  uniformly on compact subsets of $\R$; cf.\ \cite[Theorem 7]{Bourdaud:2002tb}.
\end{enumerate}

\Cref{lefttranslationLip} will follow from \Cref{prop:Lipsuf} and \Cref{prop:Lipnec}.

\subsection{Sufficiency}

We first show that $f \in \cZ^{m+k,1}(\R)$ implies that $f_*$ acts on $\La_{m+1}(\R^d)$ and is of class $\cC^{k-1,1}$.
Or approach is based on a version of \Cref{thm:main2} for maps between Banach spaces (analogous to \Cref{cor:Banach});
see \cite{FK88,KM97}. This allows for a simple lucid proof.
Similar results hold for $f_*$ acting on H\"older--Lipschitz spaces; see e.g. \cite[Theorem 2.14]{NenningRainer16}.

\begin{proposition} \label{prop:Lipsuf}
  Let $m,k \in \N_{\ge 1}$ and $f \in \cZ^{m+k,1}(\R)$.
  Then $f_*$ acts on $\La_{m+1}(\R^d)$ and $f_*: \La_{m+1}(\R^d) \rightarrow \La_{m+1}(\R^d)$
  is of class $\cC^{k-1,1}$.
\end{proposition}

\begin{proof}
Let $r:= m+1$.
For simplicity let $d=1$.
For the fact that $f_* \La_{r}(\R) \subseteq \La_{r}(\R)$ we refer to \cite[Theorem 1]{Bourdaud:2002tb}, but
see also \Cref{rem:acts}.
By \cite[Theorem 4.3.27]{FK88} or \cite[Corollary 15]{FF89} (i.e.\ the Lipschitz analogue of \Cref{cor:Banach}),
it suffices to check that $f_* : \La_{r}(\R) \to \La_{r}(\R)$ maps $\cC^\infty$-curves to $\cC^{k-1,1}$-curves.
That $t \mapsto  g(t, \cdot )$ is $\cC^\infty$ in $\La_r(\R)$
means that, for all $\ell \in \N$,  $\|\p_1^\ell  g(t, \cdot )\|_{\La_r}$ is locally bounded in $t$ (cf.\ \cite[4.1.19]{FK88}).

We first prove the case $k=1$.
Set $h(t,x) := f(g(t,x))$.
In the following we denote, for clarity,
by $h_t$, $g_t$, etc.\ the partial derivatives with respect to $t$
and write $\p_x$ for partial derivatives with respect to $x$.

Then,
\[
  h(t,x) - h(s,x) = \int_s^t h_t(\ta,x)\, d \ta = \int_s^t (f'\o g)(\ta,x) g_t(\ta,x)\, d \ta
\]
and, for $\ell \le m$,
\begin{align*}
  \p_x^\ell h(t,x) - \p_x^\ell h(s,x)
  &=
  \sum_{j=0}^\ell \binom{\ell}{j} \int_s^t \p_x^{j} (f'\o g)(\ta,x)  \p_x^{\ell-j} g_t(\ta,x)  \, d \ta.
\end{align*}
By Fa\`a di Bruno's formula,
\begin{align*}
  \p_x^{j}  (f'\o g)(\ta,x)
  &=
  \sum_{i=1}^{j} \sum_{\ga \in \Ga(i,j)}
       c_{\ga}\, (f^{(i+1)} \o g)(\ta,x) \p_x^{\ga_1}g(\ta,x) \cdots \p_x^{\ga_l} g(\ta,x), \quad (j \ge 1),
\end{align*}
where $\Ga(i,j):= \{\ga \in (\N_{\ge 1})^i : |\ga| = j\}$ and $c_\ga := \frac{j!}{i! \ga!}$,
it is readily checked that $t \mapsto h(t, \cdot)$ is locally Lipschitz into
$\cC^{m}_b(\R) := \{u \in \cC^{m}(\R) : \sup_{\ell \le m} \|u^{(\ell)}\|_{L^\infty(\R)} <\infty \}$.

To see that $t \mapsto h(t, \cdot)$ is locally Lipschitz into $\La_{r}(\R)$ it remains to show

\begin{claim*}
  For each bounded interval $I\subseteq \R$ the set
  \begin{align*}
    \Big\{
    \frac{\De^2_v \p_x^{m} h(t,x) - \De^2_v \p_x^{m} h(s,x)}{|v||t-s|} :
    x,v \in \R,\, v \ne 0,\, s \ne t \in I
    \Big\},
  \end{align*}
  is bounded,
  where the second finite difference $\De^2_v$ acts in the $x$-variable.
\end{claim*}

In view of the above, it is enough to show that for all $0\le i \le j \le m$ and $\ga \in \Ga(i,j)$,
\begin{equation} \label{eq:claim}
      \De^2_v \big[(f^{(i+1)} \o g)(\ta,x) \p_x^{\ga_1}g(\ta,x) \cdots \p_x^{\ga_l} g(\ta,x) \p_x^{m-j} g_t(\ta,x) \big] = O(|v|),
\end{equation}
uniformly in $x \in \R$ and $\ta \in I$.
Let us from now on  suppress the dependence on $\ta$ in the notation.
We can assume that $|v|$ is small;
otherwise the result follows from the fact that $t \mapsto h(t, \cdot)$ is locally Lipschitz
into $\cC^{m}_b(\R)$.

Each of the factors in the product
\begin{equation} \label{eq:product}
   (f^{(i+1)} \o g)(x) \p_x^{\ga_1}g(x) \cdots \p_x^{\ga_l} g(x) \p_x^{m-j} g_t(x)
\end{equation}
is globally bounded in $x$, locally in $\ta$.
Thus, by the product rule \eqref{eq:Leibniz}, in order to prove \eqref{eq:claim}
it suffices to show the following two facts.
\begin{description}
  \item[Fact 1] For each factor $\Pi$ in the product \eqref{eq:product} we have $\De^2_v \Pi = O(|v|)$, uniformly in $x \in \R$ and $\ta \in I$.
  \item[Fact 2] For any two factors $\Pi_1$ and $\Pi_2$ in the product \eqref{eq:product} we have $\De^1_v \Pi_1 \cdot \De^1_v \Pi_2 = O(|v|)$, uniformly in $x \in \R$ and $\ta \in I$.
\end{description}

\subsubsection*{Fact 1}
For $\Pi = \p_x^{\ell}g(x)$ and $\Pi = \p_x^{\ell}g_t(x)$, where $\ell \le m$,
the assertion $\De^2_v \Pi = O(|v|)$
holds; either by assumption if $\ell = m$
or using \eqref{eq:integral} if $\ell<m$.
It remains to consider $\Pi =(f^{(\ell+1)} \o g)(x)$ for $\ell \le m$.
In order to estimate
$\De^2_v (f^{(m+1)} \o g)(x)$
we have, by \eqref{chainrule}, to deal with terms of the form
\begin{equation} \label{eq:terms}
   \De^1_{\De^2_v g(x)} f^{(m+1)}(y) \quad \text{ and }\quad \De^2_{\De^1_v g(x)} f^{(m+1)}(y),
\end{equation}
where $y$ ranges over a bounded set.
By assumption and \eqref{eq:inclusions}, $f^{(m+1)} \in \cZ^{0,1}(\R) \subseteq \cC^{0,\al}(\R)$ for all $\al \in (0,1)$
so that
\begin{align*}
   \De^1_{\De^2_v g(x)} f^{(m+1)}(y) = O(|\De^2_v g(x)|^\al).
\end{align*}
Because $g \in \La_{r}(\R) \hookrightarrow \La_{1+\be}(\R)$ for any $\be \in (0,1)$ (as $r\ge 2$), \Cref{thm:Krantz} implies
\begin{align*}
  \De^2_v g(x) = O(|v|^{1+\be}).
\end{align*}
Taking $\al := (1+\be)^{-1}$ we conclude that
\begin{align*}
   \De^1_{\De^2_v g(x)} f^{(m+1)}(y) = O(|v|).
\end{align*}
For the second term in \eqref{eq:terms} we have
\begin{align*}
   \De^2_{\De^1_v g(x)} f^{(m+1)}(y) = O(|\De^1_v g(x)|) = O(|v|),
\end{align*}
since $g$ is globally Lipschitz.
For $\ell < m$ use \eqref{eq:integral} and similar arguments.

\subsubsection*{Fact 2}
By \eqref{chainrule} and since $f^{(m+1)} \in \cZ^{0,1}(\R) \subseteq \cC^{0,\om}(\R)$ where $\om(t):=t \log \frac{1}{t}$, we
find
\begin{align*}
\De^1_v (f^{(m+1)} \o g)(x) &= \De^1_{\De^1_v g(x)} f^{(m+1)}(g(x)) = O(\om(|\De^1_v g(x)|)) = O(\om(|v|))
\end{align*}
as well as $\De^1_v \p_x^{m} g(x) = O(\om(|v|))$ and $\De^1_v \p_x^{m}g_t(x) = O(\om(|v|))$.
If the order of differentiation in $x$ is lower than $m$, then all these terms are actually $O(|v|)$.
In any case it follows that $\De^1_v \Pi_1 \cdot \De^1_v \Pi_2 = O(|v|)$.

\medskip
This ends the proof for $k=1$.
Now we argue by induction on $k$.
Let $k>1$ and $f \in \cZ^{m+k,1}(\R)$.
Then $t \mapsto h_t(t,\cdot) = (f' \o g(t,\cdot))g_t(t,\cdot)$
is of class $\cC^{k-2,1}$ into $\La_r(\R)$, since $(f')_*$ is of class $\cC^{k-2,1}$
by induction hypothesis.
Consequently, $t \mapsto h(t,\cdot)$ is of class $\cC^{k-1,1}$ into $\La_r(\R)$ (cf.\ \cite[Theorem 4.3.24]{FK88}).
\end{proof}

\begin{remark} \label{rem:acts}
  It is not difficult to
  build a proof of the fact that $f_* \La_{m+1}(\R) \subseteq  \La_{m+1}(\R)$ if $f \in \cZ^{m,1}(\R)$
  from the arguments used above.
  To see that, conversely, $f_* \La_{m+1}(\R) \subseteq  \La_{m+1}(\R)$ implies $f \in \cZ^{m,1}(\R)$  consider $f\o g$, where $g$ is a $\cC^\infty$-function
  with compact support and $g(x)=x$ on a compact interval.
\end{remark}

\subsection{Necessity}

\begin{proposition} \label{prop:Lipnec}
  Let $m,k \in \N_{\ge 1}$ and $f : \R \to \R$ a function.
  Suppose that $f_*$ acts on $\La_{m+1}(\R)$ and $f_*: \La_{m+1}(\R) \rightarrow \La_{m+1}(\R)$
  is of class $\cC^{k-1,1}$. Then $f \in \cZ^{m+k,1}(\R)$.
\end{proposition}

\begin{proof}
    By \Cref{rem:acts}, $f \in \cZ^{m,1}(\R)$, in particular, $f \in \cC^{m}(\R)$.

    For any compact interval
    $I\subseteq \R$ let $\rh_I: \R \to \R$ be a $\cC^\infty$-function with compact support such that $\rh_I(x) = x$
    for all $x \in I$.
    Then $g(x) := \rh_I(x)$ belongs to $\La_{m+1}(\R)$
    and $c(t) := g + \rh_{[-1,1]}(t)$ defines a $\cC^\infty$-curve in $\La_{m+1}(\R)$ with $c(0)=g$ and $c(t)=g + t$ if $t \in [-1,1]$.
    By assumption, $f_*$ maps $\cC^\infty$-curves in $\La_{m+1}(\R)$ to $\cC^{k-1,1}$-curves.
    Thus $f_* (c)$ is a $\cC^{k-1,1}$-curve in $\La_{m+1}(\R)$
    and $(f_* (c))(t)(x) = f(x + t)$ if $x \in I$ and $t \in [-1,1]$.
    Thus, by the $\Hol^{k-1}_1$-version of \Cref{thm:ZcharE}, see also \cite[Lemma 12.4]{KM97},
    \begin{align*}
         \frac{\de^k_{\on{eq}} \De^2_v (f_*(c))^{(m)}(x;t)}{v}  = \frac{\De^k_t \De^2_v f^{(m)}(x)}{t^k v}
    \end{align*}
    is bounded for all $x \in I$ and all small $v,t \in \R \setminus \{0\}$. For $t = v$ we see that
    \begin{align*}
         \frac{\De^{k+2}_t f^{(m)}(x)}{t^{k+1}} = t \de^{k+2}_{\on{eq}} f^{(m)}(x;t)
    \end{align*}
    is bounded for all $x \in I$ and all small $t\ne 0$.
    Since $\La_{m+1}(\R) \hookrightarrow \on{Lip}_m(\R)$,
    $f_*(c)$ also is a $\cC^{k-1,1}$-curve in $\on{Lip}_m(\R)$ whence similar considerations give that
    \begin{align*}
        \frac{\De^k_t \De^1_v f^{(m-1)}(x)}{t^k v}
    \end{align*}
    is bounded for all $x \in I$ and all small $v,t$.
    We see that
    $\de^{k+1}_{\on{eq}} f^{(m-1)}(x;t)$
    is bounded for $x \in I$ and small $t$. Invoking \Cref{thm:Lchar} twice, we find that
    $f^{(m-1)} \in \cC^{k,1}(\R)$ which entails $f \in \cC^{m+ k -1,1}(\R)$, and so
    $\de^{k}_{\on{eq}} f^{(m)}(x;t)$ is locally bounded in
    $(x,t) \in \R \times (\R\setminus \{0\})$.
    By \Cref{thm:Zchar}, we have $f^{(m)} \in \cZ^{k,1}(\R)$ and consequently $f \in \cZ^{m+k,1}(\R)$.
\end{proof}

\subsection*{Acknowledgments}
We thank the anonymous referee for helpful remarks.


\begin{thebibliography}{10}

\bibitem{Boman67}
J.~Boman, \emph{Differentiability of a function and of its compositions with
  functions of one variable}, Math. Scand. \textbf{20} (1967), 249--268.

\bibitem{Bourdaud:2002tb}
G.~Bourdaud and M.~Lanza de~Cristoforis, \emph{{Functional Calculus in
  H{\"o}lder-Zygmund Spaces}}, {T}rans. {A}mer. {M}ath. {S}oc. \textbf{354}
  (2002), no.~10, 4109--4129.

\bibitem{LlaveObaya99}
R.~de~la Llave and R.~Obaya, \emph{Regularity of the composition operator in
  spaces of {H}{\"o}lder functions}, Discrete Contin. Dynam. Systems \textbf{5}
  (1999), no.~1, 157--184.

\bibitem{Faure89}
C.-A. Faure, \emph{Sur un th\'eor\`eme de {B}oman}, C. R. Acad. Sci. Paris
  S\'er. I Math. \textbf{309} (1989), no.~20, 1003--1006.

\bibitem{Faure91}
\bysame, \emph{Th\'eorie de la diff\'erentiation dans les espaces convenables},
  Ph.D. thesis, Universit\'e de Gen\'eve, 1991.

\bibitem{FF89}
C.-A. Faure and A.~Fr{\"o}licher, \emph{H\"older differentiable maps and their
  function spaces}, Categorical topology and its relation to analysis, algebra
  and combinatorics (Prague, 1988), World Sci. Publ., Teaneck, NJ, 1989,
  pp.~135--142.

\bibitem{FK88}
A.~Fr{\"o}licher and A.~Kriegl, \emph{Linear spaces and differentiation
  theory}, Pure and Applied Mathematics (New York), John Wiley \& Sons Ltd.,
  Chichester, 1988, A Wiley-Interscience Publication.

\bibitem{Krantz83}
S.G. {Krantz}, \emph{{Lipschitz spaces, smoothness of functions, and
  approximation theory}}, {Expo. Math.} \textbf{1} (1983), 193--260.

\bibitem{Kriegl97}
A.~Kriegl, \emph{Remarks on germs in infinite dimensions}, Acta Math. Univ.
  Comenian. (N.S.) \textbf{66} (1997), no.~1, 117--134.

\bibitem{KM97}
A.~Kriegl and P.~W. Michor, \emph{The convenient setting of global analysis},
  Mathematical Surveys and Monographs, vol.~53, American Mathematical Society,
  Providence, RI, 1997, \url{http://www.ams.org/online\_bks/surv53/}.

\bibitem{KMR}
A.~Kriegl, P.~W. Michor, and A.~Rainer, \emph{{M}any parameter {H}\"older
  perturbation of unbounded operators}, Math. Ann. \textbf{353} (2012),
  519--522.

\bibitem{NenningRainer16}
D.~N. Nenning and A.~Rainer, \emph{On groups of {H}\"older diffeomorphisms and
  their regularity}, {T}rans. {A}mer. {M}ath. {S}oc. \textbf{370} (2018),
  no.~8, 5761--5794.

\bibitem{Norton94}
A.~Norton, \emph{The {Z}ygmund {M}orse-{S}ard theorem}, J. Geom. Anal.
  \textbf{4} (1994), no.~3, 403--424.

\bibitem{Rainer18}
A.~Rainer, \emph{Arc-smooth functions on closed sets}, Compos. Math.
  \textbf{155} (2019), 645--680.

\bibitem{Rainer:2021tr}
\bysame, \emph{Arc-smooth functions and cuspidality of sets},  (2021),
  ar{X}iv:2112.14163.

\end{thebibliography}

\def\cprime{$'$}
\providecommand{\bysame}{\leavevmode\hbox to3em{\hrulefill}\thinspace}
\providecommand{\MR}{\relax\ifhmode\unskip\space\fi MR }
\providecommand{\MRhref}[2]{%
  \href{http://www.ams.org/mathscinet-getitem?mr=#1}{#2}
}
\providecommand{\href}[2]{#2}

\end{document}